\documentclass[12pt,reqno]{amsart}
\usepackage{amsmath, amssymb, amsthm} 
\usepackage{url}
\usepackage[breaklinks]{hyperref}

\renewcommand{\(}{\left\(}
\renewcommand{\)}{\right\)}
\renewcommand{\[}{\left\[}
\renewcommand{\]}{\right\]}

\numberwithin{equation}{section}
 \theoremstyle{plain}
\newtheorem{theorem}{Theorem}[section]
\newtheorem{lemma}[theorem]{Lemma}
\newtheorem{remark}[]{Remark}

\newtheorem{definition}[theorem]{Definition}
\newtheorem{corollary}[theorem]{Corollary}

   \makeatletter
\def\proof{\@ifnextchar[{\@oproof}{\@nproof}}
\def\@oproof[#1][#2]{\trivlist\item[\hskip\labelsep\textit{#2 Proof of\
#1.}~]\ignorespaces}
\def\@nproof{\trivlist\item[\hskip\labelsep\textit{Proof.}~]\ignorespaces}

\makeatother

\begin{document}
\title[Analogues of Harglotz-Zagier-Novikov function]{Analogues of Harglotz-Zagier-Novikov function}

\author{Diksha Rani Bansal}
\address{Diksha Rani\\ Department of Mathematics \\
Indian Institute of Technology Indore \\
Indore, Simrol, Madhya Pradesh, 453552, India.} 
\email{dikshaba1233@gmail.com,  msc2003141006@iiti.ac.in}

\author{Bibekananda Maji}
\address{Bibekananda Maji\\ Discipline of Mathematics \\
Indian Institute of Technology Indore \\
Indore, Simrol, Madhya Pradesh, 453552, India.} 
\email{bibek10iitb@gmail.com, bibekanandamaji@iiti.ac.in}

\author{Pragya Singh}
\address{Pragya Singh \\ Mathematical and Physical Sciences \\ School of Arts and Sciences, Ahmedabad
University \\ 
Ahmedabad, Gujarat, 380009, India.} 
\email{pragysingh1923@gmail.com}


\maketitle

{\it Dedicated to Kalyan Chakraborty and Kotyada Srinivas for their 60th birthdays}

\begin{abstract}
Recently, Choie and Kumar extensively studied the Herglotz-Zagier-Novikov function $\mathcal{F}(z;u,v)$, defined as 
\begin{align*}
\mathcal{F}(z;u,v) = \int_{0}^{1} \frac{\log(1-ut^z)}{v^{-1}-t} dt, \quad \textup{for}  \,\,\,\, \mathfrak{Re}(z)>0.
\end{align*}
They obtained two-term,  three-term and six-term functional equations for $\mathcal{F}(z;u,v)$ and also evaluated special values in terms of di-logarithmic functions.  Motivated from their work,  we study the following two integrals, 
\begin{align*}
\mathcal{F}(z;u,v,w) &=\int_{0}^1 \frac{\log(1-ut^z)\log(1-wt^z)}{v^{-1}-t}\text{d}t,  \\
\mathcal{F}_k(z;u,v) &= \int_{0}^{1} \frac{\log^k(1-ut^z)}{v^{-1}-t} \, \text{d}t,
\end{align*}
for $\mathfrak{Re}(z)>0$ and $k \in \mathbb{N}$.  
For $k=1$,  the integral $\mathcal{F}_k(z;u,v)$ reduces to $\mathcal{F}(z;u,v)$. This allows us to recover the properties of $\mathcal{F}(z;u,v)$ by studying the properties of $\mathcal{F}_k(z;u,v)$. We evaluate special values of these two functions in terms of poly-logarithmic functions.
\end{abstract}

\section{Introduction}
Finding an explicit formula for the constant term in the Laurent series expansion at $s=1$ of Dedekind zeta function is a challenging problem. Kronecker's limit formula answers this problem for Dedekind zeta functions over quadratic imaginary number fields in terms of the logarithmic of Dedekind eta function.   Zagier's \cite{DZ1975} ground breaking exploration of the Kronecker limit formula for real quadratic fields has ignited considerable interest among number theorists.  Novikov \cite{AN1981}, building upon Zagier's work introduced a novel function within the Kronecker limit formula paradigm.  Recently, Choie and Kumar \cite{CK2023} studied the properties of this function which they named as Herglotz-Zagier-Novikov function defined as
\begin{align}\label{HZN}
\mathcal{F}(z;u,v) = \int_{0}^{1} \frac{\log(1-ut^z)}{v^{-1}-t} dt, \quad \textup{for}  \,\,\,\, \mathfrak{Re}(z)>0,
\end{align}
where $ u \in \mathbb{L} := \mathbb{C} \setminus \{(1, \infty) \cup \{0\}\}$ and $v \in \mathbb{L}' := \mathbb{C} \setminus \{[1, \infty) \cup \{0\}\}$. 

This function serves as a unified framework encompassing three distinct functions extensively studied by Herglotz \cite{GH1923},  Zagier \cite{DZ1975}, and Muzaffar Williams \cite{MW2002}. One of these functions is $J(z)$ which interestingly  turns out to be a particular instance of $\mathcal{F}(z;u,v)$. For $\mathfrak{Re}(z) > 0$, we have
\begin{align*}
J(z) = -\mathcal{F}(z;-1,-1) = \int_{0}^{1} \frac{\log(1+t^z)}{1+t}\text{d}t.
\end{align*}
Radchenko and Zagier \cite{RZ23} significantly studied this function in connection with Stark’s conjecture.  
Choie and Kumar \cite{CK2023} further studied attributes of $J(z)$ and gave the following two term functional equation for $J(z)$.
For $\mathfrak{Re}(z)>0$, they showed that
\begin{align}
J(z)+J\left(\frac{1}{z}\right)=\log^2(z). \label{TTJ(z)}
\end{align}
We obtained a relation (see Theorem \ref{ext J(z)}) which allows us to evaluate $J(z)$ in the left half plane. 
We define $\mathbb{D}=\{ z \in \mathbb{C} \backslash \{0\} : |z| \leq 1 \}$ and $\mathbb{D'}=\mathbb{D} \backslash \{1\}$ throughout the paper. Choie and Kumar \cite[Theorem 2.2]{CK2023} also gave a duplication formula for $\mathcal{F}(z;u,v)$ as given below. For $\mathfrak{Re}(z)>0$, $u \in \mathbb{D}$ and $v, v^2 \in \mathbb{D'}$, 
\begin{align}
\mathcal{F}(2z;u^2,v) &= \mathcal{F}(z;u,v) + \mathcal{F}(z;-u,v), \label{eq:dup1}\\
\mathcal{F}\left(\frac{z}{2};u,v^2\right) &= \mathcal{F}(z;u,v) + \mathcal{F}(z;u,-v). \label{eq:dup2}
\end{align}
Further, they \cite[Theorem 2.5]{CK2023} obtained the following exact evaluation of $\mathcal{F}\left(z;u,v\right)$ at rational argument, for $(u,v) \in \mathbb{D}\times \mathbb{D}'$ and $m, n \in \mathbb{N}$,
\begin{align}\label{CK at p/q}
\mathcal{F}\left(\frac{m}{n};u,v\right) = \frac{n}{m}\textup{Li}_2 (u) + \sum_{\alpha^m = 1}\sum_{\beta^n = 1} \left\{\textup{Li}_2 \left(\frac{\beta v^{\frac{1}{n}}}{\beta v^{\frac{1}{n}} - 1}\right) - \textup{Li}_2 \left(\frac{\alpha u^{\frac{1}{m}}-\beta v^{\frac{1}{n}}}{1-\beta v^{\frac{1}{n}}}\right) \right\}.
\end{align}

Now we define an analogue of $\mathcal{F}(z;u,v)$ by  multiplying another $\text{log}$ term in the numerator of $\mathcal{F}(z;u,v)$, that is, 
\begin{align}\label{AHZN-1}
\mathcal{F}(z;u,v,w) :=\int_{0}^1 \frac{\log(1-ut^z)\log(1-wt^z)}{v^{-1}-t}\text{d}t, \quad  \mathfrak{Re}(z) > 0.  
\end{align}
Further, we consider a one-variable generalization of $\mathcal{F}(z;u,v)$ as follows:   
\begin{align}\label{AHZN}
\mathcal{F}_k(z;u,v) &:= \int_{0}^{1} \frac{\log^k(1-ut^z)}{v^{-1}-t} \, \text{d}t, \quad  \mathfrak{Re}(z) > 0, \quad k \in \mathbb{N}. 
\end{align}
Here $ u,w \in \mathbb{L}$ and $v \in \mathbb{L}'$ in both the cases. For $k = 1$, $\mathcal{F}_k(z;u,v)$, becomes $\mathcal{F}(z;u,v)$. Also, for $u=w$, we have $\mathcal{F}(z;u,v,u) = \mathcal{F}_2(z;u,v)$. One of main objectives of this  paper is to study the properties of $\mathcal{F}(z;u,v,w)$ and $\mathcal{F}_k(z;u,v)$.

Before delving deeper, we define the polylogarithm function which will be used throughout the paper.
\begin{definition}[{\bf Polylogarithm function}]
The polylogarithm function, denoted as $\textup{Li}_{s}(z)$, represents a fundamental mathematical concept with a dual representation: it can be expressed both as a power series in $z$ and as a Dirichlet series in $s$. It is defined as
\begin{align}
\textup{Li}_{s}(z) := \sum_{n=1}^{\infty}\frac{z^n}{n^s}.
\end{align}
This definition holds true for every complex $s$ and for any complex argument $z$ where $|z| < 1$; analytic continuation allows it to be extended to $|z| \geq 1$. 
\end{definition}

\section{Main Results}
In this section, we discuss the key results of this paper. Equations \eqref{eq:dup1}, \eqref{eq:dup2} illustrate the duplication formula for $\mathcal{F}(z;u,v)$ provided by Choie and Kumar. Inspired by this, we obtained a generalization of the duplication formula of $\mathcal{F}(z;u,v)$.

\begin{theorem} \label{Dup_HZN_gen}
Let $n$ be any natural number and $\mathfrak{Re}(z)>0$. We have
\begin{align}
 \mathcal{F}(nz;u^n,v)&=\sum_{\alpha^n=1}\mathcal{F}(z;u\alpha,v),  \quad u^n \in \mathbb{D}, v\in \mathbb{D}', \label{Special Value_1}\\
 \mathcal{F}\left(\frac{z}{n};u,v^n\right)&=\sum_{\alpha^n=1}\mathcal{F}(z;u,v\alpha), \quad u \in \mathbb{D}, v^n \in \mathbb{D}'. \label{Special Value_2}
\end{align}
\end{theorem}

\begin{remark}
Putting $n=2$ in Theorem \ref{Dup_HZN_gen}, we recover the duplication formula \eqref{eq:dup1} and \eqref{eq:dup2}.
\end{remark}

Now we mention a result for the Herglotz function $J(z)$ which gives a connection between the values in the right half plane and the left half plane.
\begin{theorem}\label{ext J(z)}
For $\mathfrak{Re}(z)>0$, we have
 $$ J(-z)=J(z)+\frac{z \pi^2}{12}.$$
\end{theorem}

Now we study the various properties of $\mathcal{F}(z;u,v,w)$. Recall that the function $\mathcal{F}(z;u,v,w)$ is defined for  $\mathfrak{Re}(z) > 0$.  Here,  we extend its domain by analytic continuation.
\begin{theorem}[Analytic Continuation of $\mathcal{F}(z;u,v,w)$]\label{Analytic Continuation_1}
For non-zero complex numbers $u, v$, and $w$ such that $|u|<1$, $|v|<1$, and $|w|<1$,
\begin{align}
\mathcal{F}(z;u,v,w)= \mathop{\sum^{\infty}\sum^{\infty}\sum^{\infty}}_{m=1\ n=1 \ k=1}\frac{u^m w^n v^k}{mn(mz+nz+k)}.
\end{align}
This holds for any complex $z$ such that $p z + q \neq 0$ for any $ p,  q \in \mathbb{N}$.   
\end{theorem}
 
Next, we find the following multiplication relations for $\mathcal{F}(z;u,v,w)$ similar to Theorem \ref{Dup_HZN_gen} for $\mathcal{F}(z;u,v)$.

\begin{theorem}[Multiplication formula]\label{Dup: analogue}
For $n \in \mathbb{N}$ and $\mathfrak{Re}(z) > 0$,  we have
\begin{align}
 \mathcal{F}(nz;u^n,v,w^n) &= \sum_{\alpha^n=1}\sum_{\beta^n=1}\mathcal{F}(z;u\alpha,v,w\beta), \quad u^n, w^n \in \mathbb{D}, v \in \mathbb{D}', \label{special Value_4} \\
\mathcal{F}\left(\frac{z}{n};u,v^n,w\right) &= \sum_{\alpha^n=1}\mathcal{F}(z;u,v\alpha,w), \quad u, w \in \mathbb{D}, v^n \in \mathbb{D'}. \label{special Value_5}
\end{align}
\end{theorem}

In the next section, we explore some particular values of $\mathcal{F}(z;u,v,w)$ for $ z \in \mathbb{Q}$.

\subsection{Special evaluations of $\mathcal{F}(z;u,v,w)$}
For $z=1$ and $u=w$ in \eqref{special Value_5}, we have the following result.
\begin{corollary}\label{anal at 1/n}
Let $n$ be any natural number. Let $ u, v^n \in \mathbb{D'} $  such that $u \neq v,$ we have
\begin{align}
\mathcal{F}\left(\frac{1}{n};u,v^{n},u\right)&=\sum_{\alpha^n=1}\bigg( -\log^2(1-u) \log\left(\frac{u(1-\alpha v)}{u-\alpha v}\right)-2\log(1-u) \textup{Li}_2 \left(\frac{\alpha v(u-1)}{u-\alpha v}\right)\nonumber\\
&+2\textup{Li}_3\left(\frac{\alpha v(u-1)}{u-\alpha v}\right)-2\textup{Li}_3 \left(\frac{\alpha v}{\alpha v-u}\right)\bigg).
\end{align}
\end{corollary}

\begin{remark} Letting $u$ tends to $1^-$ in Corollary \ref{anal at 1/n}, one can check that the first three terms will go to zero. Therefore, we get
\begin{align}\label{special value at u=1 and 1/n}
\mathcal{F}\left(\frac{1}{n};1,v^n,1\right) = -\sum_{\alpha^n=1}2 \textup{Li}_3\left(\frac{\alpha v}{\alpha v-1}\right).
\end{align}
One can evaluate the value of $\mathcal{F} \left(\frac{1}{n};1,-1,1\right)$ for any odd natural number $n$ directly using \eqref{special value at u=1 and 1/n}.
In particular, for $n=1$, we get
\begin{align*}
\mathcal{F}\left(1;1,-1,1\right) = -2\,\operatorname{Li}_{3}\left(\frac{1}{2}\right).
 \end{align*}
 It is interesting to note that $\operatorname{Li}_{3}\left(\frac{1}{2}\right) = \frac{7}{8}\zeta(3) - \frac{\pi^2}{12}\log (2) + \frac{\log^3(2)}{6}$, see \cite[Eq.~(1.14)]{Lewin}.
\end{remark}
The next result gives an explicit evaluation of $\mathcal{F}(1;u,v,w)$ when $u = v = w$.

\begin{theorem}\label{for u=v=w}
 For any $u \in \mathbb{D}'$, we have
\begin{align}
\mathcal{F}\left(1;u,u,u\right)&= - \frac{1}{3} \log^3 \left(1-u \right).
\end{align}
\end{theorem}
Now for $ u \neq v \neq w$, we give an evaluation of $\mathcal{F}(1;u,v,w)$ which will be used further for calculating $\mathcal{F}(z;u,v,w)$ at rational arguments of $z$.
\begin{theorem}\label{Final lemma}
Let $u, v, w, \in \mathbb{D}'$ such that $ u \neq v \neq w$. We have
\begin{align}
\mathcal{F}(1;u,v,w) &= \log\left(1-u\right)\log\left(\frac{u(1-v)}{u-v}\right)\log\left(\frac{v}{v-w}\right) - \log\left(1-w\right) \operatorname{Li}_{2}\left(\frac{v(1-u)}{v-u}\right) \nonumber \\
&  - \log(1-u)\log\left(\frac{w(1-v)}{w-v}\right)\log\left(\frac{v(1-w)}{v-w}\right) - \log(1-u)\operatorname{Li}_{2}\left(\frac{v(1-w)}{v-w}\right) \nonumber \\
& - \log\left(\frac{v-u}{v-w}\right)\bigg[\operatorname{Li}_{2}\left(\frac{v(w-u)}{w(v-u)}\right) - \log \left(\frac{w}{w-u}\right)\log \left(\frac{u(w-v)}{w(u-v)}\right) \bigg] \nonumber \\
& + \log\left(\frac{(v-u)(w-1)}{(v-w)(u-1)}\right)\bigg[\operatorname{Li}_{2}\bigg(\frac{(v-u)(w-1)}{(v-w)(u-1)}\bigg) + \operatorname{Li}_{2}\left(\frac{v(1-u)}{v-u}\right) \nonumber \\
& -\operatorname{Li}_2\left(\frac{v(1-w)}{v-w}\right) -\operatorname{Li}_{2}\left(\frac{u(1-w)}{w(1-u)}\right)\bigg] \nonumber \\ 
& + \frac{1}{2}\log\left(\frac{w(u-v)}{u(w-v)}\right)\bigg[\log^{2} \left( \frac{u-v}{v(u-1)}\right) - \log^2\left(\frac{v}{v-u}\right)\bigg] \nonumber \\
& + \operatorname{Li}_3\left(\frac{v(1-w)}{v-w}\right)-\operatorname{Li}_3\left(\frac{v}{v-w}\right) +\operatorname{Li}_3\left(\frac{v(1-u)}{v-u}\right)-\operatorname{Li}_3\left(\frac{v}{v-u}\right) \nonumber \\
& + \operatorname{Li}_{3}\left(\frac{u(1-w)}{w(1-u)}\right) - \operatorname{Li}_{3}\left(\frac{(v-u)(w-1)}{(v-w)(u-1)}\right) - \operatorname{Li}_{3}\left(\frac{u}{w}\right) + \operatorname{Li}_{3}\left(\frac{v-u}{v-w}\right). \nonumber
\end{align}

\end{theorem}

\begin{theorem}\label{val at rational for anlo}
Let $u, v, w \in \mathbb{D}'$ such that $u \neq v \neq w$ and $p,q \in \mathbb{N}$. Then $\mathcal{F}\left(\frac{p}{q};u,v,u\right)$ can be given as
\begin{align}
\mathcal{F}\left(\frac{p}{q};u,v,w\right)&=\sum_{\alpha^q=1} \sum_{\beta^p=1} \sum_{\gamma^p=1} \mathcal{F}(1;u^\frac{1}{p}\beta,v^\frac{1}{q}\alpha,w^\frac{1}{p}\gamma).
\end{align}
where the value $\mathcal{F}(1;u, v, w)$ is provided in Theorem \ref{Final lemma}.

\end{theorem}

We can evaluate $\mathcal{F}\left(\frac{p}{q};u,v,w\right)$ on two different arguments using the above theorem, one by putting $p=n, q=1$ and another with $p=1, q=n$.

\begin{corollary}\label{anal at natural}
For any natural number $n$ and $u, v, w \in \mathbb{D'}$ such that $u \neq v \neq w$, we have
\begin{align}
\mathcal{F}(n;u,v,w) &= \sum_{\beta^n=1}\sum_{\gamma^n=1}\mathcal{F}(1;u^\frac{1}{n}\beta,v,w^\frac{1}{n}\gamma), \\
\mathcal{F}\left(\frac{1}{n};u,v,w\right) &= \sum_{\alpha^n=1}\mathcal{F}(1;u,v^\frac{1}{n}\alpha,w).
\end{align} 
\end{corollary}

\subsection{The function $\mathcal{F}_k(z;u,v)$}
We now shift our focus to the function $\mathcal{F}_k(z;u,v)$ defined in \eqref{AHZN} and examine its characteristics.
First, we give an analytic continuation of $\mathcal{F}_k(z;u,v)$ to $\mathbb{C}$ except at negative rationals.
\begin{theorem}\label{theorem:analytic_continuation}
For $|u|<1$ and $|v|<1$, we have
\begin{align}
\mathcal{F}_k(z;u,v) &= \prod_{i=1}^{k}\sum_{\substack{m_i, l=1}}^{\infty}\frac{u^{m_i} v^l}{m_i(z(m_1+m_2+ \cdots +m_k) +l)}, \quad z \neq -\frac{p}{q}\quad \text{where } p,q \in \mathbb{N}.\nonumber
\end{align}
\end{theorem}
The next result gives one of the variants of multiplication formula.

\begin{theorem}\label{theorem:duplication_formula}
For $u \in \mathbb{D}$, $v^n\in \mathbb{D'}$, and $\mathfrak{Re}(z)>0$, then for any given natural number $n$, we have
\begin{align}
\mathcal{F}_k\left(\frac{z}{n};u,v^n\right)&= \sum_{\beta^n=1}\mathcal{F}_k(z;u,\beta v). \label{special Value_3}
\end{align}
\end{theorem}

In the next theorem, we give the explicit evaluation of $\mathcal{F}_k\left(\frac{1}{n};u,v\right)$ for any $n \in \mathbb{N}$ .
\begin{theorem}\label{GEN at 1/n}
For any $u,v \in \mathbb{D}'$ such that $ u \neq v$,
\begin{align}
\mathcal{F}_k\left(\frac{1}{n};u,v\right)&=\sum_{\beta^n=1}\bigg\{\sum_{j=1}^{k+1} (-1)^{j-1} \log^{k+1-j}(1-u) \operatorname{Li}_j\left(\frac{\beta v^\frac{1}{n}(u-1)}{u-\beta v^\frac{1}{n}}\right)\frac{k!}{(k+1-j)!}\nonumber\\
& +(-1)^{k+1} k! \operatorname{Li}_{k+1}\left(\frac{\beta v^\frac{1}{n}}{\beta v^\frac{1}{n}-u}\right)\bigg\}.
\end{align}
\end{theorem}
Substituting $k=1$ in the above theorem, we recover the formula \eqref{CK at p/q} with $m=1$ given by Choie and Kumar. 
\begin{corollary}\label{Particular for GEN at 1/n}
For $(u,v) \in \mathbb{D} \times \mathbb{D}'$ with $ u \neq v$, we have
\begin{align}\label{RC at 1/n}
\mathcal{F}\left(\frac{1}{n};u,v\right) = n\textup{Li}_2 (u) + \sum_{\beta^n = 1} \left\{\textup{Li}_2 \left(\frac{\beta v^{\frac{1}{n}}}{\beta v^{\frac{1}{n}} - 1}\right) - \textup{Li}_2 \left(\frac{u-\beta v^{\frac{1}{n}}}{1-\beta v^{\frac{1}{n}}}\right) \right\}.
\end{align}
\end{corollary}

Letting $u$ tends to $1^-$ in Theorem \ref{GEN at 1/n}, one can see that the first term will vanish and hence we obtain the result below.
\begin{corollary}
For any $ n \in \mathbb{N}$ and $v \in \mathbb{D}'$, we have
\begin{align}\label{GEN at u=1 and 1/n}
\mathcal{F}_{k}\left(\frac{1}{n};1,v\right) &= (-1)^{k+1}\, k! \sum_{\beta^n=1} \operatorname{Li}_{k+1}\left(\frac{\beta v^\frac{1}{n}}{\beta v^\frac{1}{n}-1}\right).
\end{align}
In particular, when $n=1$, $v=-1$ in \eqref{GEN at u=1 and 1/n}, we get
\begin{align}
\mathcal{F}_{k}\left(1;1,-1\right) &= (-1)^{k+1}\, k! \, \operatorname{Li}_{k+1}\left(\frac{1}{2}\right).
\end{align}
\end{corollary}
\begin{remark}
For $k \geq 4$, an explicit evaluation of $\operatorname{Li}_{k+1}\left(\frac{1}{2}\right)$ in terms of well-known functions is challenging. In 1995, Borwein et. al. \cite{Borwein} showed that $$\operatorname{Li}_{4}\left(\frac{1}{2}\right)= \frac{\pi^4}{360} -\frac{1}{24} \log^4(2) + \frac{\pi^2}{24} \log^2(2) -\frac{1}{2} \zeta(\bar{3}, \bar{1}),
$$
where $\zeta(\bar{3}, \bar{1})= \sum_{m>n=1}^\infty (-1)^{m+n} m^{-3} n^{-1}$. More generally, Broadhurst \cite{Broadhurst} expressed $\operatorname{Li}_{k}\left(\frac{1}{2}\right)$ in terms of special value of the multiple zeta function, namely, $$\operatorname{Li}_{k}\left(\frac{1}{2}\right) = - \zeta(\bar{1}, \bar{1}, \{1\}^{k-2}) = - \sum_{n_k> \dots >n_2>n_1=1}^\infty \frac{(-1)^{n_1+n_2} }{n_1 n_2 \cdots n_k},$$
where $\{1\}^{k-2}$ denotes $(k-2)$-tuple with all entries equal to $1$.  
\end{remark}

Theorem \ref{GEN at 1/n} does not hold for the case $u = v$. So, when $u=v$, we have the following result.
\begin{theorem}\label{u=v and x=1}
 For any $u \in \mathbb{D'}$, we have
\begin{align}
\mathcal{F}_k\left(1;u,u\right)&=-\frac{1}{1+k} \log^{1+k} \left(1-u \right).
\end{align}
\end{theorem}

\begin{remark}
For $k=2$, the above result matches with Theorem \ref{for u=v=w} since $\mathcal{F}_2\left(1;u,u\right) = \mathcal{F}\left(1;u,u, u\right).$
\end{remark}

\section{Some Important Lemmas}
In this section, we give a few important lemmas that will help us to prove our main results. To calculate the value of $\mathcal{F}(z;u,v,w)$ at $z= \frac{p}{q}$, we need the value of $\mathcal{F}(1;u,v,w)$ where 
\begin{align*}
\mathcal{F}(1;u,v,w) :=\int_{0}^1 \frac{\log(1-ut)\log(1-wt)}{v^{-1}-t}\text{d}t.
\end{align*}
To evaluate the above integral, we use the following Lemmas along with the below properties of polylog function:
\begin{align}
\operatorname{Li}_{1}(z) &= -\log(1-z), \label{log to Li1}\\
\frac{\mathrm{d}}{\mathrm{d}z}\operatorname{Li}_{s+1}(z) &= \frac{\operatorname{Li}_s(z)}{z}. \label{RBL}
\end{align}

\begin{lemma}\label{lemma 1} 
Let $u, v, w \in \mathbb{D}'$ such that $u \neq v, \,\,  v \neq w$. Then, we have
\begin{align*}
I(u, v, w) & := \int_{0}^1 \log^2 \left(\frac{(v-u)(wt-1)}{(v-w)(ut-1)}\right)\frac{v}{(vt-1)} \mathrm{d}t  \\
&= \log^2 \left(\frac{(v-u)(w-1)}{(v-w)(u-1)}\right) \log\left(\frac{w(v-1)}{v-w}\right) -\log^2\left(\frac{v-u}{v-w}\right) \log\left(\frac{w}{w-v}\right) \\
& + 2\log \left(\frac{(v-u)(w-1)}{(v-w)(u-1)}\right)\bigg[\operatorname{Li}_{2}\left(\frac{(v-u)(w-1)}{(v-w)(u-1)}\right) -\operatorname{Li}_{2}\left(\frac{u(1-w)}{w(1-u)}\right)\bigg] \\
& - 2\log\left(\frac{v-u}{v-w}\right)\bigg[\operatorname{Li}_{2}\left(\frac{v-u}{v-w}\right)-\operatorname{Li}_{2}\left(\frac{u}{w}\right)\bigg]+2\operatorname{Li}_{3}\left(\frac{u(1-w)}{w(1-u)}\right)\\
&-2\operatorname{Li}_{3}\left(\frac{(v-u)(w-1)}{(v-w)(u-1)}\right)-2\operatorname{Li}_{3}\left(\frac{u}{w}\right)+2\operatorname{Li}_{3}\left(\frac{v-u}{v-w}\right).
\end{align*}

\end{lemma}

\begin{proof} First,  we add and subtract $ \frac{u}{u t-1}$ in the integrand to write as  
\begin{align*}
I(u, v, w) &= \int_{0}^1 \log^2 \left(\frac{(v-u)(wt-1)}{(v-w)(ut-1)}\right)\bigg[\frac{v}{(vt-1)}-\frac{u}{(ut-1)}\bigg]\text{d}t\\
& - \int_{0}^1 \log^2 \left(\frac{(v-u)(wt-1)}{(v-w)(ut-1)}\right)\frac{u}{(1-ut)} \text{d}t.
\end{align*}
Now, we use integration by parts to solve the above integral. In both integrals, we consider $\log^2 \left(\frac{(v-u)(wt-1)}{(v-w)(ut-1)}\right)$ as the first function and we use the following relations for the case of second function: 
\begin{align*}
\frac{\mathrm{d}}{\mathrm{d}t} \log \left(\frac{(u-w)(vt-1)}{(ut-1)(v-w)}\right) & = \frac{v}{(vt-1)}-\frac{u}{(ut-1)},   \\
\frac{\mathrm{d}}{\mathrm{d}t} \log \left(\frac{u-w}{w(ut-1)}\right) & = \frac{u}{(1-ut)}.
\end{align*}
Hence we have
{\allowdisplaybreaks
\begin{align*}
I(u, v, w) &= \log^2 \left(\frac{(v-u)(w-1)}{(v-w)(u-1)}\right)\log\left(\frac{(u-w)(v-1)}{(u-1)(v-w)}\right) - \log^2 \left(\frac{v-u}{v-w}\right)\log\left(\frac{u-w}{v-w}\right)\\
& - 2\int_{0}^1 \log \left(\frac{(v-u)(wt-1)}{(v-w)(ut-1)}\right)\log\left(\frac{(u-w)(vt-1)}{(ut-1)(v-w)}\right)\frac{u-w}{(ut-1)(vt-1)} \mathrm{d}t \\
& - \log^2 \left(\frac{(v-u)(w-1)}{(v-w)(u-1)}\right)\log\left(\frac{u-w}{w(u-1)}\right) + \log^2 \left(\frac{v-u}{v-w}\right)\log\left(\frac{u-w}{w}\right)\\
& + 2\int_{0}^1 \log \left(\frac{(v-u)(wt-1)}{(v-w)(ut-1)}\right)\log\left(\frac{u-w}{w(ut-1)}\right)\frac{u-w}{(ut-1)(wt-1)} \mathrm{d}t.
\end{align*}}
We have two integrals in the last equation. To solve the first integral, we use the relation \eqref{log to Li1} and \eqref{RBL} with $z = \frac{(v-u)(wt-1)}{(v-w)(ut-1)}$ to get
$$\frac{\mathrm{d}}{\mathrm{d}t} \operatorname{Li}_2 \left(\frac{(v-u)(wt-1)}{(v-w)(ut-1)}\right)  = \frac{w-u}{(ut-1)(wt-1)}\log \left(\frac{(u-w)(vt-1)}{(ut-1)(v-w)}\right).$$ 
For the second integral, we again use \eqref{log to Li1} and \eqref{RBL} with $z= \frac{u(wt-1)}{w(ut-1)}$ to obtain
$$\frac{\mathrm{d}}{\mathrm{d}t} \operatorname{Li}_2 \left(\frac{u(wt-1)}{w(ut-1)}\right)  = \frac{w-u}{(ut-1)(wt-1)}\log \left(\frac{u-w}{w(ut-1)}\right),$$
and again use integration by parts. We also simplify the log terms to have
\begin{align}
I(u, v, w) &= \log^2 \left(\frac{(v-u)(w-1)}{(v-w)(u-1)}\right) \log\left(\frac{w(v-1)}{v-w}\right) -\log^2\left(\frac{v-u}{v-w}\right) \log\left(\frac{w}{w-v}\right) \nonumber \\
& + 2\log \left(\frac{(v-u)(w-1)}{(v-w)(u-1)}\right)\operatorname{Li}_{2}\left(\frac{(v-u)(w-1)}{(v-w)(u-1)}\right) -2\log \left(\frac{v-u}{v-w}\right)\operatorname{Li}_{2}\left(\frac{v-u}{v-w}\right) \nonumber \\
&  - 2\int_{0}^1 \frac{u-w}{(ut-1)(wt-1)}\operatorname{Li}_{2} \left(\frac{(v-u)(wt-1)}{(v-w)(ut-1)}\right)\mathrm{d}t \nonumber \\
& + 2\log \left(\frac{(v-u)(w-1)}{(v-w)(u-1)}\right)\operatorname{Li}_{2}\left(\frac{u(w-1)}{w(u-1)}\right) - 2\log \left(\frac{v-u}{v-w}\right)\operatorname{Li}_{2}\left(\frac{u}{w}\right) \nonumber \\
& - 2 \int_{0}^1 \frac{u-w}{(ut-1)(wt-1)}\operatorname{Li}_{2}\left(\frac{u(wt-1)}{w(ut-1)}\right)\text{d}t. \label{Last integral lemma 3.1}
\end{align}
Using \eqref{RBL}, one can verify that
\begin{align*}
\frac{\mathrm{d}}{\mathrm{d}t} \operatorname{Li}_3 \left(\frac{(v-u)(wt-1)}{(v-w)(ut-1)}\right) &= \frac{u-w}{(ut-1)(wt-1)}\operatorname{Li}_2 \left(\frac{(v-u)(wt-1)}{(v-w)(ut-1)}\right),\\
\frac{\mathrm{d}}{\mathrm{d}t} \operatorname{Li}_3 \left(\frac{u(wt-1)}{w(ut-1)}\right) &= \frac{u-w}{(ut-1)(wt-1)}\operatorname{Li}_2 \left(\frac{u(wt-1)}{w(ut-1)}\right).
\end{align*}
Use these values in \eqref{Last integral lemma 3.1} to obtain the final result.
\end{proof}

\begin{lemma}\label{lemma 2}
Let $u, v, w, \in \mathbb{D}'$ such that $ u \neq v \neq w$. We define
\begin{align}
J(u, v, w) := \int_{0}^1 \frac{u}{ut-1}\operatorname{Li}_2\left(\frac{v(1-wt)}{v-w}\right)\mathrm{d}t. \label{J(u,v,w)}
\end{align}
Then,  we have
\begin{align*}
J(u, v, w) + & J(w, v, u)\\
& = \operatorname{Li}_3\left(\frac{v(1-w)}{v-w}\right)-\operatorname{Li}_3\left(\frac{v}{v-w}\right) +\operatorname{Li}_3\left(\frac{v(1-u)}{v-u}\right)-\operatorname{Li}_3\left(\frac{v}{v-u}\right) \\
&+\log\left(\frac{v-u}{v-w}\right)\bigg[\operatorname{Li}_{2}\left(\frac{v}{v-w}\right)+\log\left(\frac{w(u-v)}{u(w-v)}\right) \log\left(\frac{v}{v-w}\right) - \operatorname{Li}_2\left(\frac{v}{v-u}\right)\bigg] \\
&-\log\left(\frac{(v-u)(w-1)}{(v-w)(u-1)}\right)\bigg[\operatorname{Li}_2\left(\frac{v(1-w)}{v-w}\right)+\log\left(\frac{w(u-v)}{u(w-v)}\right)\log\left(\frac{v(1-w)}{v-w}\right) \\
& - \operatorname{Li}_2\left(\frac{v(1-u)}{v-u}\right) \bigg] +\frac{1}{2}\log\left(\frac{w(u-v)}{u(w-v)}\right)\bigg[ \log^{2} \left( \frac{v(1-w)}{v-w}\right)-\log^2\left(\frac{v}{v-w}\right)\bigg]\\
&-\frac{1}{2}\bigg[\log\left(\frac{u(1-v)}{u-v}\right)\log^2\left(\frac{(v-u)(w-1)}{(v-w)(u-1)}\right)-\log^2\left(\frac{v-u}{v-w}\right)\log\left(\frac{u}{u-v}\right)\bigg]\nonumber \\
& - \log\left(\frac{w(u-v)}{u(w-v)}\right)\bigg[\log\left(\frac{v(1-w)}{v-w}\right)\log\left(\frac{w(1-u)}{w-u}\right) - \log\left(\frac{v}{v-w}\right)\log\left(\frac{w}{w-u}\right) \nonumber \\
& + \operatorname{Li}_2 \left(\frac{u(1-w)}{u-w}\right) - \operatorname{Li}_2 \left(\frac{u}{u-w}\right)\bigg] +  \frac{1}{2}I(u, v, w),
\end{align*}
where $I(u,  v,  w)$ as in Lemma \ref{lemma 1}.
\end{lemma}

\begin{proof}
First,  we write
\begin{align*}
J(u, v, w) &= \int_{0}^1\frac{w}{wt-1}\operatorname{Li}_2\left(\frac{v(1-wt)}{v-w}\right)\text{d}t\\
&-\int_{0}^1 \left(\frac{w}{wt-1} - \frac{u}{ut-1} \right)\operatorname{Li}_2\left(\frac{v(1-wt)}{v-w}\right)\text{d}t.
\end{align*}
To solve the first integral, we take $z = \frac{v(1-wt)}{v-w}$ in \eqref{RBL} to see that 
\begin{align*}
\frac{\mathrm{d}}{\mathrm{d}t} \operatorname{Li}_3 \left(\frac{v(1-wt)}{v-w}\right) &= \frac{w}{wt-1}\operatorname{Li}_2 \left(\frac{v(1-wt)}{v-w}\right).
\end{align*}
For the second integral, we use integration by parts, by taking $\operatorname{Li}_2\left(\frac{v(1-wt)}{v-w}\right)$ as first function and using the following relations:
\begin{align*}
\frac{\mathrm{d}}{\mathrm{d}t} \log\left(\frac{(v-u)(wt-1)}{(v-w)(ut-1)}\right)  &=  \frac{w}{(wt-1)} - \frac{u}{(ut-1)}, \\
\frac{\mathrm{d}}{\mathrm{d}t} \operatorname{Li}_2 \left(\frac{v(1-wt)}{v-w}\right) &= \frac{w}{1-wt}\log\left(\frac{w(1-vt)}{w-v}\right).
\end{align*}
So, we have
\begin{align}
J(u,  v,  w) & = \operatorname{Li}_3\left(\frac{v(1-w)}{v-w}\right)-\operatorname{Li}_3\left(\frac{v}{v-w}\right)+\log\left(\frac{v-u}{v-w}\right)  \operatorname{Li}_{2}\left(\frac{v}{v-w}\right) \nonumber  \\
& -\log\left(\frac{(v-u)(w-1)}{(v-w)(u-1)}\right)\operatorname{Li}_2\left(\frac{v(1-w)}{v-w}\right)  + J_1(u,  v,  w),  \label{J with J1}
\end{align}
where
\begin{align}\label{definition J_1}
J_1(u,  v,  w) :=  \int_{0}^1 \frac{w}{1-wt}\log\left(\frac{w(1-vt)}{w-v}\right)  \log\left(\frac{(v-u)(wt-1)}{(v-w)(ut-1)}\right) \text{d}t.
\end{align}
Now to simplify $ J_1(u,  v,  w),$ we again use integration by parts with the following relation:
\begin{align*}
\frac{\mathrm{d}}{\mathrm{d}t} \log \left(\frac{v(1-wt)}{v-w}\right) &= \frac{w}{wt-1}.
\end{align*}
Thus, one has
\begin{align}
J_1(u,  v,  w) & = - \log\left(\frac{(v-u)(w-1)}{(v-w)(u-1)}\right) \log\left(\frac{w(1-v)}{w-v}\right)\log\left(\frac{v(1-w)}{v-w}\right) \nonumber \\
& + \log\left(\frac{v-u}{v-w}\right)\log\left(\frac{w}{w-v}\right)\log\left(\frac{v}{v-w}\right) \nonumber \\
& + J_2(u,  v,  w) + J_3(u,  v,  w),  \label{J_1}
\end{align}
where 
\begin{align}
& J_2(u,  v,  w) := \int_{0}^1 \frac{v}{vt-1}\log\left(\frac{(v-u)(wt-1)}{(v-w)(ut-1)}\right)\log \left(\frac{v(1-wt)}{v-w}\right)\text{d}t,  \\
& J_3(u,  v,  w) :=  \int_{0}^1 \left[\frac{w}{wt-1}-\frac{u}{ut-1}\right]  \log\left(\frac{w(1-vt)}{w-v}\right)\log\left(\frac{v(1-wt)}{v-w}\right)\text{d}t.
\end{align}
We solve $J_2(u, v, w)$ using the following relation:
 \begin{align}
\frac{\mathrm{d}}{\mathrm{d}t} \log\left(\frac{u(1-vt)}{u-v}\right) = \frac{v}{vt-1}. \label{v integral}
\end{align}
Hence, we have
\begin{align*}
J_2(u,  v,  w) &= \log\left(\frac{(v-u)(w-1)}{(v-w)(u-1)}\right)\log \left(\frac{v(1-w)}{v-w}\right)\log\left(\frac{u(1-v)}{u-v}\right) \\
& - \log\left(\frac{v-u}{v-w}\right)\log \left(\frac{v}{v-w}\right)\log\left(\frac{u}{u-v}\right) \\
& - \int_{0}^1 \frac{w}{wt - 1}\log\left(\frac{(v-u)(wt-1)}{(v-w)(ut-1)}\right)\log\left(\frac{u(1-vt)}{u-v}\right) \text{d}t \\
&- \int_{0}^1 \left[\frac{w}{wt-1}-\frac{u}{ut-1}\right]\log \left(\frac{v(1-wt)}{v-w}\right)\log\left(\frac{u(1-vt)}{u-v}\right)\text{d}t.
\end{align*}
Adding the above expression of  $J_2(u,  v,  w)$ with $J_3(u,  v,  w)$ gives
\begin{align}
J_2(u,  v,  w) + J_3(u,  v,  w) &= \log\left(\frac{(v-u)(w-1)}{(v-w)(u-1)}\right)\log \left(\frac{v(1-w)}{v-w}\right)\log\left(\frac{u(1-v)}{u-v}\right) \nonumber\\
& - \log\left(\frac{v-u}{v-w}\right)\log \left(\frac{v}{v-w}\right)\log\left(\frac{u}{u-v}\right) \nonumber\\
& - \int_{0}^1 \frac{w}{wt-1}\log\left(\frac{(v-u)(wt-1)}{(v-w)(ut-1)}\right)\log\left(\frac{u(1-vt)}{u-v}\right) \text{d}t \nonumber\\
& + \log\left(\frac{w(u-v)}{u(w-v)}\right)\int_{0}^1 \left[\frac{w}{wt-1}-\frac{u}{ut-1}\right] \log\left(\frac{v(1-wt)}{v-w}\right)\text{d}t.\label{J2 + J3}
\end{align}
One can check that the following term in \eqref{J2 + J3} simplifies as 
\begin{align}
\int_{0}^1\frac{w}{wt-1}\log\left(\frac{v(1-wt)}{v-w}\right)\text{d}t &= \frac{1}{2}\bigg[\log^{2} \left( \frac{v(1-w)}{v-w}\right) - \log^{2} \left(\frac{v}{v-w}\right)\bigg]. \label{the last integral}
\end{align}
Again,  for solving $\int_{0}^1 \frac{u}{ut-1}\log\left(\frac{v(1-wt)}{v-w}\right)\text{d}t$ in \eqref{J2 + J3}, we use the following relation:
\begin{align*}
\frac{\mathrm{d}}{\mathrm{d}t} \log\left(\frac{w(1-ut)}{w-u}\right) = \frac{u}{ut-1}.
\end{align*}
Hence, we have
\begin{align}
\int_{0}^1 \frac{u}{ut-1}\log\left(\frac{v(1-wt)}{v-w}\right)\text{d}t &=  \log\left(\frac{v(1-w)}{v-w}\right)\log\left(\frac{w(1-u)}{w-u}\right) \nonumber\\
&- \log\left(\frac{v}{v-w}\right)\log\left(\frac{w}{w-u}\right) \nonumber \\
& + \int_{0}^1 \frac{w}{1-wt}\log\left(\frac{w(1-ut)}{w-u}\right)\text{d}t \label{the last integral1}
\end{align}
We finally use the following formula 
$$\frac{\mathrm{d}}{\mathrm{d}t} \operatorname{Li}_2 \left(\frac{u(1-wt)}{u-w}\right) = \frac{w}{1-wt}\log\left(\frac{w(1-ut)}{w-u}\right), $$
in \eqref{the last integral1} to have
\begin{align}
\int_{0}^1 \frac{u}{ut-1}\log\left(\frac{v(1-wt)}{v-w}\right)\text{d}t &= \log\left(\frac{v(1-w)}{v-w}\right)\log\left(\frac{w(1-u)}{w-u}\right) \nonumber\\
& - \log\left(\frac{v}{v-w}\right)\log\left(\frac{w}{w-u}\right) \nonumber \\
& + \operatorname{Li}_2 \left(\frac{u(1-w)}{u-w}\right) - \operatorname{Li}_2 \left(\frac{u}{u-w}\right). \label{the last integral2}
\end{align}
We also define the remaining integral in \eqref{J2 + J3} as follows:
\begin{align}
J_4(u, v, w) & := \int_{0}^1 \frac{w}{wt-1}\log\left(\frac{(v-u)(wt-1)}{(v-w)(ut-1)}\right)\log\left(\frac{u(1-vt)}{u-v}\right) \text{d}t.  \label{Definition J_4}
\end{align}
We break this integral into two parts,  namely,  
\begin{align}
J_4(u, v, w) = J_5(u, v, w)  + J_6( u,  v,  w),   \label{J_4}
\end{align}
where 
\begin{align}
J_5(u, v, w) & := \int_{0}^1 \left[\frac{w}{(wt-1)} - \frac{u}{(ut-1)}\right]\log\left(\frac{(v-u)(wt-1)}{(v-w)(ut-1)}\right)\log\left(\frac{u(1-vt)}{u-v}\right) \text{d}t, \label{J_5}\\
J_6(u, v, w) & := \int_{0}^1 \frac{u}{(ut-1)}\log\left(\frac{(v-u)(wt-1)}{(v-w)(ut-1)}\right)\log\left(\frac{u(1-vt)}{u-v}\right) \text{d}t. \label{J_6}
\end{align}
To solve $J_5(u, v, w)$, we take product of log terms as first function and remaining terms as second function. We use the following formula:
\begin{align*}
\frac{\mathrm{d}}{\mathrm{d}t} \log \left(\frac{(v-u)(wt-1)}{(v-w)(ut-1)}\right) & = \frac{w}{(wt-1)}-\frac{u}{(ut-1)}.
\end{align*}
Finally, we get
\begin{align}
J_5(u, v, w) & = \log^2 \left(\frac{(v-u)(w-1)}{(v-w)(u-1)}\right)\log\left(\frac{u(1-v)}{u-v}\right) - \log^2 \left(\frac{v-u}{v-w}\right)\log\left(\frac{u}{u-v}\right) \nonumber \\
& - \int_{0}^1 \frac{v}{(vt-1)}\log^2 \left(\frac{(v-u)(wt-1)}{(v-w)(ut-1)}\right)\text{d}t \nonumber \\
& - \int_{0}^1 \left[\frac{w}{(wt-1)}-\frac{u}{(ut-1)}\right]\log \left(\frac{(v-u)(wt-1)}{(v-w)(ut-1)}\right)\log\left(\frac{u(1-vt)}{u-v}\right)\text{d}t \nonumber \\
& = \log^2 \left(\frac{(v-u)(w-1)}{(v-w)(u-1)}\right)\log\left(\frac{u(1-v)}{u-v}\right) - \log^2 \left(\frac{v-u}{v-w}\right)\log\left(\frac{u}{u-v}\right) \nonumber \\
& - I(u, v, w) - J_4(u, v, w) + J_6(u, v, w),  \label{J_5 solved}
\end{align}
where $I(u,v,w)$ is same as in Lemma \ref{lemma 1} and $J_4(u, v, w)$ and $J_6(u, v, w)$ are defined as in \eqref{Definition J_4}, \eqref{J_6}. To solve $J_6(u, v, w)$, we take $z = \frac{v(1-ut)}{v-u}$ in \eqref{RBL} to see that
$$\frac{\mathrm{d}}{\mathrm{d}t} \operatorname{Li}_2 \left(\frac{v(1-ut)}{v-u}\right) = -\frac{u}{(ut-1)}\log \left(\frac{u(1-vt)}{u-v}\right).$$
Using integration by parts and the above relation, we have
\begin{align}
J_6(u, v, w) & = - \log \left(\frac{(v-u)(w-1)}{(v-w)(u-1)}\right) \operatorname{Li}_2 \left(\frac{v(1-u)}{v-u}\right) + \log \left(\frac{v-u}{v-w}\right)\operatorname{Li}_2 \left(\frac{v}{v-u}\right) \nonumber \\
&  + \int_{0}^1 \left[\frac{w}{(wt-1)}-\frac{u}{(ut-1)}\right]\operatorname{Li}_2 \left(\frac{v(1-ut)}{v-u}\right)\text{d}t. \label{J_6 solved 1}
\end{align}
Now using \eqref{RBL}, we have
\begin{align}
\frac{\mathrm{d}}{\mathrm{d}t} \operatorname{Li}_3 \left(\frac{v(1-ut)}{v-u}\right) &= \frac{u}{ut-1}\operatorname{Li}_2\left(\frac{v(1-ut)}{v-u}\right). \label{Derivative of Li3}
\end{align}
Employing this relation in \eqref{J_6 solved 1}, one gets
\begin{align}
J_6(u, v, w) & = - \log \left(\frac{(v-u)(w-1)}{(v-w)(u-1)}\right) \operatorname{Li}_2 \left(\frac{v(1-u)}{v-u}\right) + \log \left(\frac{v-u}{v-w}\right)\operatorname{Li}_2 \left(\frac{v}{v-u}\right) \nonumber \\
&  -\operatorname{Li}_3\left(\frac{v(1-u)}{v-u}\right) +\operatorname{Li}_3\left(\frac{v}{v-u}\right) + \int_{0}^1 \frac{w}{(wt-1)}\operatorname{Li}_2 \left(\frac{v(1-ut)}{v-u}\right)\text{d}t. \label{J_6 solved 2}
\end{align}
Note that the last integral term is nothing but $J(w, v ,u)$.
Putting the values of $J_5(u, v, w)$ \eqref{J_5 solved} and $J_6(u,  v,  w)$  \eqref{J_6 solved 2} in \eqref{J_4}, we have
\begin{align}
J_4(u, v, w) &= \log^2 \left(\frac{(v-u)(w-1)}{(v-w)(u-1)}\right)\log\left(\frac{u(1-v)}{u-v}\right) - \log^2 \left(\frac{v-u}{v-w}\right)\log\left(\frac{u}{u-v}\right) \nonumber \\
& - I(u, v, w) - J_4(u, v, w) - 2\log \left(\frac{(v-u)(w-1)}{(v-w)(u-1)}\right) \operatorname{Li}_2 \left(\frac{v(1-u)}{v-u}\right) \nonumber \\
& + 2\log \left(\frac{v-u}{v-w}\right)\operatorname{Li}_2 \left(\frac{v}{v-u}\right) - 2\operatorname{Li}_3\left(\frac{v(1-u)}{v-u}\right) +2\operatorname{Li}_3\left(\frac{v}{v-u}\right)\nonumber \\
&  + 2\int_{0}^1 \frac{w}{(wt-1)}\operatorname{Li}_2 \left(\frac{v(1-ut)}{v-u}\right)\text{d}t.
\end{align}
The last integral in the above equation is nothing but $J(w, v, u)$. Thus, simplifying further one has
\begin{align}
J_4(u, v, w) &= \frac{1}{2}\bigg[\log^2 \left(\frac{(v-u)(w-1)}{(v-w)(u-1)}\right)\log\left(\frac{u(1-v)}{u-v}\right) - \log^2 \left(\frac{v-u}{v-w}\right)\log\left(\frac{u}{u-v}\right) \bigg] \nonumber \\
&-\frac{1}{2}I(u, v, w) - \log \left(\frac{(v-u)(w-1)}{(v-w)(u-1)}\right) \operatorname{Li}_2 \left(\frac{v(1-u)}{v-u}\right) + \log\left(\frac{v-u}{v-w}\right)\operatorname{Li}_2\left(\frac{v}{v-u}\right) \nonumber \\
& -\operatorname{Li}_3\left(\frac{v(1-u)}{v-u}\right) + \operatorname{Li}_3\left(\frac{v}{v-u}\right) + J(w, v, u). \label{J_4 solved}
\end{align}
Substituting the above value of $J_4(u,  v,  w)$ and \eqref{the last integral}, \eqref{the last integral2} in \eqref{J2 + J3}, and then together with \eqref{J_1} and \eqref{J with J1}, we get the final answer.
\end{proof}

\begin{lemma}\label{f11}
For any $ u ,v \in \mathbb{L}' $ such that $u \neq v $, we have
\begin{align}
\mathcal{F}_k(1;u,v)&:= \int_{0}^1 \frac{\log^k(1-ut)}{v^{-1}-t}\text{d}t\nonumber\\
&=\sum_{j=1}^{k+1}\bigg\{ (-1)^{j-1} \log^{k+1-j}(1-u) \operatorname{Li}_j\left(\frac{v(u-1)}{u-v}\right)\frac{k!}{(k+1-j)!}\bigg\} \nonumber \\
& +(-1)^{k+1} k! \operatorname{Li}_{k+1}\left(\frac{v}{v-u}\right).
\end{align}
\end{lemma}

\begin{proof}
From the definition \eqref{AHZN}, we have
\begin{align*}
\mathcal{F}_k(1;u,v)& = \int_{0}^1 \frac{\log^k(1-ut)}{v^{-1}-t}\text{d}t.
\end{align*}
Now we proceed by considering $\log^k(1-ut)$ as the first function and use \eqref{v integral} to get
\begin{align*}
\mathcal{F}_k(1;u,v) &= -\log^k(1-u) \log\left(\frac{u(1-v)}{u-v}\right)-k \int_{0}^1 \log^{k-1}(1-ut) \frac{u}{1-ut} \log\left(\frac{u(1-vt)}{u-v}\right)\text{d}t.
\end{align*}
Further utilizing the following relation 
$$\frac{\mathrm{d}}{\mathrm{d}t} \operatorname{Li}_2 \left(\frac{v(1-ut)}{v-u}\right) = \frac{u}{(1-ut)}\log \left(\frac{u(1-vt)}{u-v}\right),$$
we get,
\begin{align*}
\mathcal{F}_k(1;u,v) &= -\log^k(1-u) \log\left(\frac{u(1-v)}{u-v}\right)-k\log^{k-1}(1-u) \textup{Li}_2 \left(\frac{v(1-u)}{v-u}\right)\\
& - k(k-1) \int_{0}^1 \log^{k-2}(1-ut) \frac{u}{1-ut} \textup{Li}_2 \left(\frac{v(1-ut)}{v-u}\right)\text{d}t
\end{align*}
To solve the integral in the above equation, we use following relation
$$\frac{\mathrm{d}}{\mathrm{d}t} \operatorname{Li}_3 \left(\frac{v(1-ut)}{v-u}\right) = -\frac{u}{(1-ut)}\operatorname{Li}_2 \left(\frac{v(1-ut)}{v-u}\right).$$
So, we have
\begin{align*}
\mathcal{F}_k(1;u,v) &= -\log^k(1-u) \log\left(\frac{u(1-v))}{u-v}\right)-k\log^{k-1}(1-u) \textup{Li}_2 \left(\frac{v(1-u)}{v-u}\right)\\
& + k(k-1)\log^{k-2}(1-u) \textup{Li}_3\left(\frac{v(1-u)}{v-u}\right) \\
& + k(k-1)(k-2) \int_{0}^1 \log^{k-3}(1-ut) \frac{u}{1-ut} \textup{Li}_3\left(\frac{v(1-ut)}{v-u}\right)\text{d}t.
\end{align*}
Continuing this process by taking $z= \frac{v(1-ut)}{v-u}$ in \eqref{RBL}, we get
\begin{align*}
\mathcal{F}_k(1;u,v) &= -\log^k(1-u)\log\left(\frac{u(1-v)}{u-v}\right)-k\log^{k-1}(1-u)\textup{Li}_2 \left(\frac{v(1-u)}{v-u}\right)\nonumber \\
&+k(k-1)\log^{k-2}(1-u) \textup{Li}_3\left(\frac{v(1-u)}{v-u}\right)+\cdots + (-1)^{k-1} k!\, \log(1-u) \textup{Li}_{k}\left(\frac{v(1-u)}{v-u}\right) \\
&+ (-1)^{k+1} k! \int_0^1 \frac{u}{1-ut} \textup{Li}_k\left(\frac{v(1-ut)}{v-u}\right)\text{d}t.\\
\end{align*}
Using \eqref{log to Li1}, we have
\begin{align*}
\mathcal{F}_k(1;u,v) &= \log^k(1-u)\textup{Li}_1\left(\frac{v(1-u)}{v-u}\right)-k\log^{k-1}(1-u)\textup{Li}_2 \left(\frac{v(1-u)}{v-u}\right)\nonumber \\
&+k(k-1)\log^{k-2}(1-u) \textup{Li}_3\left(\frac{v(1-u)}{v-u}\right)+\cdots + (-1)^{k-1} k!\, \log(1-u) \textup{Li}_{k}\left(\frac{v(1-u)}{v-u}\right) \\
& + (-1)^{k} k! \, \textup{Li}_{k+1}\left(\frac{v(1-u)}{v-u}\right) + (-1)^{k+1} k! \, \textup{Li}_{k+1}\left(\frac{v}{v-u}\right)\nonumber\\
&=\sum_{j=1}^{k+1} (-1)^{j-1} \log^{k+1-j}(1-u) \textup{Li}_j\left(\frac{v(1-u)}{v-u}\right)\frac{k!}{(k+1-j)!} \\
& + (-1)^{k+1}\, k!\, \textup{Li}_{k+1}\left(\frac{v}{v-u}\right).\nonumber
\end{align*}
This finishes the proof.
\end{proof}
\begin{remark} Note that the function $\mathcal{F}_2(z;u,v)$ is same  as $\mathcal{F}(1;u,v,w)$ when $u=w$. Therefore, from Lemma \ref{f11}, we have
\begin{align}
\mathcal{F}(1;u,v,u) = &-\log^2(1-u) \log\left(\frac{u(1-v)}{u-v}\right)-2\log(1-u) \textup{Li}_2 \left(\frac{v(u-1)}{u-v}\right)+2\textup{Li}_3\left(\frac{v(u-1)}{u-v}\right)\nonumber\\
&-2\textup{Li}_3 \left(\frac{v}{v-u}\right).  \label{at u=w and any v}
\end{align}

\end{remark}

\section{Proof of Main Results}
This section focuses on presenting the proofs of the main results outlined in this paper.
\begin{proof}[Theorem \rm{\ref{Special Value_1}}][]

Replacing $z$ by $nz$ and $u$ by $u^n$ in \eqref{HZN}, one has
\begin{align}\label{z by nz}
\mathcal{F}(nz;u^n,v) &= \int_{0}^1 \frac{\log(1-u^nt^{nz})}{v^{-1}-t}\text{d}t.
\end{align}
We know that
\begin{align}\label{nth roots}
1-z^n = \prod_{\alpha^n = 1}(1-\alpha z).
\end{align}
Here the product runs over all the $n$th roots of unity.
Utilizing \eqref{nth roots} in \eqref{z by nz}, we see that
\begin{align*}
\mathcal{F}(nz;u^n,v)
&=\int_{0}^1\frac{\sum_{\alpha^n=1}\log(1-\alpha ut^z)}{v^{-1}-t}\text{d}t.
\end{align*}
On swapping the order of summation and integration, we directly get \eqref{Special Value_1}.
Again from \eqref{HZN}, we have
\begin{align}
\mathcal{F}\left(\frac{z}{n};u,v^n\right) &= \int_{0}^1 \frac{\log\left(1-ut^{\frac{z}{n}}\right)}{v^{-n}-t} \, \text{d}t. \nonumber 
\end{align}
Making the change of variable $t \rightarrow e^{-t}$ in the above equation, one obtains
\begin{align}
\mathcal{F}\left(\frac{z}{n};u,v^n\right) &= \int_{0}^{\infty}\frac{\log\left(1-ue^{\frac{-tz}{n}}\right)}{e^{t}v^{-n}-1} \,\text{d}t. \nonumber
\end{align}
Substituting $t=ny$, we see that 
\begin{align}
\mathcal{F}\left(\frac{z}{n};u,v^n\right) &=  n \int_{0}^{\infty}\frac{\log\left(1-ue^{-z y}\right)}{e^{ ny }v^{-n}-1} \,\text{d}y. \nonumber
\end{align}
Again, we change the variable as $y=-\log t$ to obtain 
\begin{align}
\mathcal{F}\left(\frac{z}{n};u,v^n\right) &=  n \int_{0}^{1}\frac{\log\left(1-u t^{z} \right)}{(vt)^{-n}-1} \, \frac{\text{d}t}{t}. \nonumber
\end{align}
Now using the relation 
\begin{align}\label{ABCD}
\frac{n}{1- Y^n}=\sum_{\beta^n=1}\frac{1}{1-\beta Y}, \quad \textrm{with}~~ Y=(vt)^{-1},
\end{align}
and upon simplification, we obtain
\begin{align}
\mathcal{F}\left(\frac{z}{n};u,v^n\right) = \sum_{\beta^n=1}\mathcal{F}(z;u,\beta v).\nonumber
\end{align}
This proves \eqref{Special Value_2}.
\end{proof}

\begin{proof}[Theorem \rm{\ref{ext J(z)}}][]
Substituting $z$ by $-z$ in $J(z)$, we get
\begin{align}
J(-z) &= \int_{0}^{1} \frac{\log(1+t^{-z})}{1+t}\text{d}t \nonumber \\
&= \int_{0}^{1} \frac{\log(1+t^{z})}{1+t}\text{d}t - \int_{0}^{1} \frac{\log(t^{z})}{1+t}\text{d}t \nonumber \\
&= \int_{0}^{1} \frac{\log(1+t^{z})}{1+t}\text{d}t - z\int_{0}^{1} \frac{\log(t)}{1+t}\text{d}t. \label{J(-z)}
\end{align}
It is easy to check that 
\begin{align}
\int_{0}^{1} \frac{\log(t)}{1+t}\text{d}t = \textup{Li}_2(-1) = -\frac{\pi^2}{12}. \label{Li(-1)}
\end{align}
Substituting \eqref{Li(-1)} in \eqref{J(-z)}, we get
\begin{align}
J(-z)&= J(z)- \textup{Li}_2(-1)z \nonumber\\
\Rightarrow J(-z)&=J(z)+\frac{\pi^2z}{12}. \nonumber
 \end{align}
This completes the proof Theorem \ref{ext J(z)}.
\end{proof}

\begin{proof}[Theorem \rm{\ref{Analytic Continuation_1}}][]
As $|u|<1$, $|v|<1$, and $|w|<1$, so using the series expansion of logarithm and $\frac{1}{1-vt}$ around $t=0$, we write 
\begin{align*}
\mathcal{F}(z;u,v,w) &= \int_{0}^{1} \sum_{m=1}^{\infty}\frac{(ut^z)^m}{m} \sum_{n=1}^{\infty}\frac{(wt^z)^n}{n} \sum_{k=0}^{\infty} v^{(k+1)}t^k \, \text{d}t \nonumber \\
&= \sum_{m=1}^{\infty}\frac{u^m}{m}\sum_{n=1}^{\infty}\frac{w^n}{n} \sum_{k=0}^{\infty} (v^{k+1})\frac{1}{(mz+zn+k+1)} \nonumber\\
&= \sum_{m=1}^{\infty}\sum_{n=1}^{\infty}\sum_{k=1}^{\infty} \frac{u^m w^n v^k}{mn(z(m+n)+k)}; \quad z \neq -\frac{p}{q}
  \,\,\, \text{for any}\hspace{1.2mm} p, q \in \mathbb{N}.
\end{align*}
This series represents an analytic function in $u,v,w$. It decays for the specified values of $u$, $v$, and $w$, suggesting that the series is uniformly convergent for any complex $z$ other than a negative rational number. 
 As a result, the function $\mathcal{F}(z;u,v,w)$ is analytic for any complex $z$ expect negative rational numbers.

\end{proof}

\begin{proof}[Theorem \rm{\ref{Dup: analogue}}][]
From \eqref{AHZN-1}, we see that
\begin{align}
\mathcal{F}\left(nz,u^n,v,w^n\right)&=\int_{0}^{1}\frac{\log\left(1-u^nt^{nz}\right)\log\left(1-w^nt^{nz}\right)}{v^{-1}-t}\text{d}t. \nonumber
\end{align}
Using \eqref{nth roots}, one obtains
\begin{align}
\mathcal{F}\left(nz,u^n,v,w^n\right) &=\int_{0}^{1} \frac{\sum_{\alpha^n=1}\log(1-u\alpha t^z)\sum_{\beta^n=1}\log(1-w\beta t^z)}{v^{-1}-t} \, \text{d}t. \nonumber
\end{align}
On interchanging summation and integration, we arrive
\begin{align}
 \mathcal{F}\left(nz,u^n,v,w^n\right) &=\sum_{\alpha^n=1}\sum_{\beta^n=1}\int_{0}^{1} \frac{\log(1-u\alpha t^z)\log(1-w\beta t^z)}{v^{-1}-t} \, \text{d}t, \nonumber \\  
 & =\sum_{\alpha^n=1}\sum_{\beta^n=1}\mathcal{F}\left(z;u\alpha,v,w\beta\right). \nonumber
\end{align}
Again use \eqref{AHZN-1} to see that
\begin{align}\label{eq_abc}
\mathcal{F}\left(\frac{z}{n};u,v^n,w\right) = \int_{0}^{1}\frac{\log\left(1-ut^\frac{z}{n}\right)\log\left(1-wt^\frac{z}{n}\right)}{v^{-n}-t}\text{d}t.
\end{align}
Making the change of variable  $t=e^{-t}$ in \eqref{eq_abc}, one obtains 
\begin{align*}
\mathcal{F}\left(\frac{z}{n};u,v^n,w\right) &= \int_{0}^{\infty}\frac{\log\left(1-ue^{\frac{-tz}{n}}\right)\log\left(1-we^{\frac{-tz}{n}}\right)}{e^{t}v^{-n}-1} \,\text{d}t.
 \end{align*}
Substituting $t=ny$ and using relation \eqref{ABCD}, we get
 \begin{align*}
 \mathcal{F}\left(\frac{z}{n};u,v^n,w\right) =\sum_{\beta^n=1}\mathcal{F}(z;u,\beta v,w).
 \end{align*}
This completes the proof of Theorem \ref{Dup: analogue}.
\end{proof}

\begin{proof}[Corollary \rm{\ref{anal at 1/n}}][]
Using multiplication formula \eqref{special Value_5} for $\mathcal{F}\left(\frac{1}{n};u,v^{n},u\right)$, we get
\begin{align}
\mathcal{F}\left(\frac{1}{n};u,v^{n},u\right)&=\sum_{\alpha^n=1}\mathcal{F}(1;u,\alpha v,u).\label{Col 2.5.1}
\end{align}
Employing \eqref{at u=w and any v} with $v$ being replaced by $\alpha v$, we have
\begin{align}
\mathcal{F}(1;u,\alpha v,u) &= -\log^2(1-u) \log\left(\frac{u(1-\alpha v)}{u-\alpha v}\right)-2\log(1-u) \textup{Li}_2 \left(\frac{\alpha v(u-1)}{u-\alpha v}\right)\nonumber\\
&+2\textup{Li}_3\left(\frac{\alpha v(u-1)}{u-\alpha v}\right)-2\textup{Li}_3 \left(\frac{\alpha v}{\alpha v-u}\right). \nonumber 
\end{align}
Substituting the above value of $\mathcal{F}(1;u,\alpha v,u)$ in \eqref{Col 2.5.1}, we get the desired result.

\end{proof}

\begin{proof}[Theorem \rm{\ref{for u=v=w}}][]
From the definition \eqref{AHZN-1}, we have
\begin{align}
\mathcal{F}\left(1;u,u,u\right)&=\int_{0}^1 \frac{ \log^2 \left( 1 - ut \right)}{ \frac{1}{u}-t}\text{d}t. \nonumber
\end{align}
On integrating, one gets
\begin{align*}
\mathcal{F}\left(1;u,u,u\right)&= -\log^3 \left( 1 - u \right)-2\mathcal{F}\left(1;u,u,u\right)\nonumber \\
\Rightarrow \mathcal{F}\left(1;u,u,u\right)&=-\frac{1}{3} \log^3 \left(1-u\right).\nonumber
\end{align*}
Hence the result is proved.
\end{proof}

\begin{proof}[Theorem \rm{\ref{Final lemma}}][]
We have 
\begin{align*}
\mathcal{F}(1;u, v, w) = \int_{0}^1 \frac{\log(1-ut) \log(1-wt)}{v^{-1}-t}\text{d}t.
\end{align*}
We now consider $\log(1-ut) \log(1-wt)$ as the first function and use integration by parts, along with the relation
$$\frac{\mathrm{d}}{\mathrm{d}t}\log\left(\frac{u(1-vt)}{u-v}\right) = -\frac{v}{1-vt}\text{d}t,$$ 
we see that
\begin{align}
\mathcal{F}(1;u,v,w) &= -\log (1-u)\log(1-w)\log\left(\frac{u(1-v)}{u-v}\right)+\int_{0}^1 \frac{u}{ut-1}\log\left(1-wt\right)\nonumber \\
&\times \log\left(\frac{u(1-vt)}{u-v}\right)\text{d}t+\int_{0}^1 \frac{w}{wt-1}\log\left(1-ut\right)\log\left(\frac{u(1-vt)}{u-v}\right)\text{d}t. \label{Lemma3 int 1}
\end{align}
We further use integration by parts to evaluate both integrals in the right hand side. We take $\log\left(1-wt\right)$ as the first function and $\frac{u}{ut-1}\log\left(\frac{u(1-vt)}{u-v}\right)$ as the second function in the first integral, whereas in the second integral, we take product of $\log$ terms as the first function and $\frac{w}{wt-1}$ as the second function. From \eqref{RBL}, we can see that
\begin{align*}
\frac{\mathrm{d}}{\mathrm{d}t} \operatorname{Li}_2\left(\frac{v(1-ut)}{v-u}\right) &= -\frac{u}{ut-1}\log \left(\frac{u(1-vt)}{u-v}\right) \\
\frac{\mathrm{d}}{\mathrm{d}t} \log\left(\frac{v(1-wt)}{v-w}\right) &= \frac{w}{wt-1}.
\end{align*}
Using the above two relations in \eqref{Lemma3 int 1}, we have
\begin{align}\label{FA}
\mathcal{F}(1;u,v,w) &= -\log\left(1-u\right)\log\left(1-w\right)\log\left(\frac{u(1-v)}{u-v}\right)-\log\left(1-w\right) \operatorname{Li}_{2}\left(\frac{v(1-u)}{v-u}\right)\nonumber \\
&+ \int_{0}^1 \frac{w}{wt-1}\operatorname{Li}_{2}\left(\frac{v(1-ut)}{v-u}\right)\text{d}t + \log\left(1-u\right)\log\left(\frac{u(1-v)}{u-v}\right)\log\left(\frac{v(1-w)}{v-w}\right)\nonumber \\
&-\int_{0}^1 \frac{v}{vt-1}\log(1-ut)\log\left(\frac{v(1-wt)}{v-w}\right)\text{d}t \nonumber\\
&- \int_{0}^1\frac{u}{ut-1}\log\left(\frac{u(1-vt)}{u-v}\right)\log\left(\frac{v(1-wt)}{v-w}\right)\text{d}t \nonumber \\
& = \log\left(1-u\right)\log\left(\frac{u(1-v)}{u-v}\right)\bigg[\log\left(\frac{v(1-w)}{v-w}\right)-\log\left(1-w\right)\bigg] \nonumber \\ 
& -\log\left(1-w\right) \operatorname{Li}_{2}\left(\frac{v(1-u)}{v-u}\right) + J(w, v, u) - S(u, v, w),
\end{align}
where 
 $J(w, v, u)$ is defined as in \eqref{J(u,v,w)}. We also define
\begin{align*}
S(u, v, w) & := \int_{0}^1 \frac{v}{vt-1}\log\left(\frac{v(1-wt)}{v-w}\right)\log(1-ut)\,\,\text{d}t \\
& + \int_{0}^1\frac{u}{ut-1}\log\left(\frac{u(1-vt)}{u-v}\right)\log\left(\frac{v(1-wt)}{v-w}\right)\text{d}t.
\end{align*}
Again using integration by parts, by taking log term as first function, $\frac{v}{vt-1}$ as second and making use of the relation
$$\frac{\mathrm{d}}{\mathrm{d}t} \log\left(\frac{w(1-vt)}{w-v}\right) = \frac{v}{vt-1}$$
in the first integral, we have
\begin{align*}
S(u, v, w) &= \log(1-u)\log\left(\frac{w(1-v)}{w-v}\right)\log\left(\frac{v(1-w)}{v-w}\right) \\
& -\int_{0}^1 \frac{u}{ut-1}\log\left(\frac{v(1-wt)}{v-w}\right)\log\left(\frac{w(1-vt)}{w-v}\right)\text{d}t \\
& -\int_{0}^1\frac{w}{wt-1}\log\left(\frac{w(1-vt)}{w-v}\right)\log(1-ut)\,\text{d}t \\ 
& + \int_{0}^1\frac{u}{ut-1}\log\left(\frac{u(1-vt)}{u-v}\right)\log\left(\frac{v(1-wt)}{v-w}\right)\text{d}t.
\end{align*}
First and third integrals can be combined and we solve the second integral by taking $z = \frac{v(1-wt)}{v-w}$ in \eqref{RBL} to have
$$\frac{\mathrm{d}}{\mathrm{d}t} \operatorname{Li}_2\left(\frac{v(1-wt)}{v-w}\right) = -\frac{w}{wt-1}\log\left(\frac{w(1-vt)}{w-v}\right).$$ 
Hence, we get
\begin{align}
S(u, v, w) &= \log(1-u)\log\left(\frac{w(1-v)}{w-v}\right)\log\left(\frac{v(1-w)}{v-w}\right) \nonumber\\
& -\log\left(\frac{w(u-v)}{u(w-v)}\right)\int_{0}^1\frac{u}{ut-1}\log\left(\frac{v(1-wt)}{v-w}\right)\text{d}t \nonumber\\
& + \log(1-u)\operatorname{Li}_{2}\left(\frac{v(1-w)}{v-w}\right)-\int_{0}^1\frac{u}{ut-1}\operatorname{Li}_{2}\left(\frac{v(1-wt)}{v-w}\right)\text{d}t \nonumber \\
&= \log(1-u)\log\left(\frac{w(1-v)}{w-v}\right)\log\left(\frac{v(1-w)}{v-w}\right) + \log(1-u)\operatorname{Li}_{2}\left(\frac{v(1-w)}{v-w}\right) \nonumber \\
& - \log\left(\frac{v(1-w)}{v-w}\right)\log\left(\frac{w(1-u)}{w-u}\right) - \log\left(\frac{w(u-v)}{u(w-v)}\right) \bigg[\log\left(\frac{v}{v-w}\right)\log\left(\frac{w}{w-u}\right) \nonumber \\
& + \operatorname{Li}_2 \left(\frac{u(1-w)}{u-w}\right) - \operatorname{Li}_2 \left(\frac{u}{u-w}\right)\bigg] - J(u, v, w) \label{S solved}
\end{align}
where the first integral is solved in \eqref{the last integral2} and $J(u, v, w)$ is defined as in \eqref{J(u,v,w)}. Putting the value of $S(u, v, w)$ \eqref{S solved} and using Lemma \ref{lemma 2} in \eqref{FA}, we have
{\allowdisplaybreaks
\begin{align}
\mathcal{F}(1;u,v,w) &= \log\left(1-u\right)\log\left(\frac{u(1-v)}{u-v}\right)\log\left(\frac{v}{v-w}\right) - \log\left(1-w\right) \operatorname{Li}_{2}\left(\frac{v(1-u)}{v-u}\right) \nonumber \\
&  - \log(1-u)\log\left(\frac{w(1-v)}{w-v}\right)\log\left(\frac{v(1-w)}{v-w}\right) - \log(1-u)\operatorname{Li}_{2}\left(\frac{v(1-w)}{v-w}\right) \nonumber \\
& - \log\left(\frac{v-u}{v-w}\right)\bigg[\operatorname{Li}_{2}\left(\frac{v-u}{v-w}\right)+ \operatorname{Li}_{2}\left(\frac{v}{v-u}\right)-\operatorname{Li}_{2}\left(\frac{v}{v-w}\right)-\operatorname{Li}_{2}\left(\frac{u}{w}\right) \nonumber \\
& -\log\left(\frac{w(u-v)}{u(w-v)}\right) \log\left(\frac{v}{v-w}\right)\bigg] \nonumber \\
& + \log\left(\frac{(v-u)(w-1)}{(v-w)(u-1)}\right)\bigg[\operatorname{Li}_{2}\bigg(\frac{(v-u)(w-1)}{(v-w)(u-1)}\bigg) +\operatorname{Li}_{2}\left(\frac{v(1-u)}{v-u}\right) -\operatorname{Li}_2\left(\frac{v(1-w)}{v-w}\right) \nonumber \\
& -\operatorname{Li}_{2}\left(\frac{u(1-w)}{w(1-u)}\right) -\log\left(\frac{w(u-v)}{u(w-v)}\right)\log\left(\frac{v(1-w)}{v-w}\right) \bigg] \nonumber \\
& +\frac{1}{2}\log\left(\frac{w(u-v)}{u(w-v)}\right)\bigg[ \log^{2} \left( \frac{v(1-w)}{v-w}\right)-\log^2\left(\frac{v}{v-w}\right)\bigg] \nonumber \\
& -\frac{1}{2}\bigg[\log\left(\frac{u(1-v)}{u-v}\right)\log^2\left(\frac{(v-u)(w-1)}{(v-w)(u-1)}\right)-\log^2\left(\frac{v-u}{v-w}\right)\log\left(\frac{u}{u-v}\right)\bigg] \nonumber \\
& + \frac{1}{2}\log^2 \left(\frac{(v-u)(w-1)}{(v-w)(u-1)}\right) \log\left(\frac{w(v-1)}{v-w}\right) - \frac{1}{2}\log^2\left(\frac{v-u}{v-w}\right) \log\left(\frac{w}{w-v}\right) \nonumber \\
& + \operatorname{Li}_3\left(\frac{v(1-w)}{v-w}\right)-\operatorname{Li}_3\left(\frac{v}{v-w}\right) +\operatorname{Li}_3\left(\frac{v(1-u)}{v-u}\right)-\operatorname{Li}_3\left(\frac{v}{v-u}\right) \nonumber \\
& + \operatorname{Li}_{3}\left(\frac{u(1-w)}{w(1-u)}\right) - \operatorname{Li}_{3}\left(\frac{(v-u)(w-1)}{(v-w)(u-1)}\right) - \operatorname{Li}_{3}\left(\frac{u}{w}\right) + \operatorname{Li}_{3}\left(\frac{v-u}{v-w}\right). \label{Rearranged eqn}
\end{align}}
We simplify some terms of the last equation. We use the below functional equation of dilogarithm function Li$_2 (z)$, due to Rogers \cite{Rogers1907}:
\begin{align}
\text{Li}_{2}(A) + \text{Li}_{2}(B) - \text{Li}_{2}(AB) &= \text{Li}_{2}\left(\frac{A-AB}{1-AB}\right) + \text{Li}_{2}\left(\frac{B-AB}{1-AB}\right) \nonumber \\
& + \log\left(\frac{1-A}{1-AB}\right)\log\left(\frac{1-B}{1-AB}\right). \label{dilog fn eqn}
\end{align}
We first take $(A, B) = \bigg(\frac{v-u}{v-w}, \frac{v}{v-u}\bigg)$ to write as
\begin{align}
\operatorname{Li}_{2}\left(\frac{v-u}{v-w}\right)+ \operatorname{Li}_{2}&\left(\frac{v}{v-u}\right)-\operatorname{Li}_{2}\left(\frac{v}{v-w}\right)-\operatorname{Li}_{2}\left(\frac{u}{w}\right) \nonumber \\
 &= \operatorname{Li}_{2}\left(\frac{v(w-u)}{w(v-u)}\right) - \log \left(\frac{w}{w-u}\right)\log \left(\frac{u(w-v)}{w(u-v)}\right). \label{dilog fn eq 1}
\end{align} 
We also have the following terms in \eqref{Rearranged eqn}:
{\allowdisplaybreaks
\begin{align}
L(u, v, w) & := \frac{1}{2}\log\left(\frac{w(u-v)}{u(w-v)}\right)\bigg[ \log^{2} \left( \frac{v(1-w)}{v-w}\right)-\log^2\left(\frac{v}{v-w}\right)\bigg] \nonumber \\
& -\frac{1}{2}\bigg[\log\left(\frac{u(1-v)}{u-v}\right)\log^2\left(\frac{(v-u)(w-1)}{(v-w)(u-1)}\right)-\log^2\left(\frac{v-u}{v-w}\right)\log\left(\frac{u}{u-v}\right)\bigg] \nonumber \\
& + \frac{1}{2}\log^2 \left(\frac{(v-u)(w-1)}{(v-w)(u-1)}\right) \log\left(\frac{w(v-1)}{v-w}\right) - \frac{1}{2}\log^2\left(\frac{v-u}{v-w}\right) \log\left(\frac{w}{w-v}\right) \nonumber \\
&= \frac{1}{2}\log\left(\frac{w(u-v)}{u(w-v)}\right)\bigg[ \log^{2} \left( \frac{v(1-w)}{v-w}\right)-\log^2\left(\frac{v}{v-w}\right)\bigg] \nonumber \\
& + \frac{1}{2}\log^2 \left(\frac{(v-u)(w-1)}{(v-w)(u-1)}\right)\bigg[\log\left(\frac{w(1-v)}{w-v}\right) - \log\left(\frac{u(1-v)}{u-v}\right)\bigg] \nonumber \\
& -\frac{1}{2}\log^2\left(\frac{v-u}{v-w}\right)\bigg[\log\left(\frac{w}{w-v}\right) - \log\left(\frac{u}{u-v}\right)\bigg] \nonumber \\
& = \frac{1}{2}\log\left(\frac{w(u-v)}{u(w-v)}\right)\bigg[\log^{2} \left( \frac{v(1-w)}{v-w}\right) + \log^2\left(\frac{(v-u)(w-1)}{(v-w)(u-1)}\right) \nonumber \\
& - \log^2\left(\frac{v}{v-w}\right) - \log^2\left(\frac{v-u}{v-w}\right)\bigg]. \label{L}
\end{align}}
 We also use $a^2 + b^2 = (a-b)^2 + 2ab$ for the $\log^2$ terms to see that 
\begin{align}
&\log^{2} \left( \frac{v(1-w)}{v-w}\right) + \log^2\left(\frac{(v-u)(w-1)}{(v-w)(u-1)}\right) - \log^2\left(\frac{v}{v-w}\right) - \log^2\left(\frac{v-u}{v-w}\right) \nonumber \\
& =  \log^{2} \left( \frac{u-v}{v(u-1)}\right) + 2\log\left( \frac{v(1-w)}{v-w}\right)\log\left(\frac{(v-u)(w-1)}{(v-w)(u-1)}\right) - \log^2\left(\frac{v}{v-u}\right) \nonumber \\
& - 2\log\left(\frac{v}{v-w}\right)\log\left(\frac{v-u}{v-w}\right). \label{log sq eqn}
\end{align}
Putting the value of \eqref{log sq eqn} in \eqref{L} and then using \eqref{dilog fn eq 1} and the value of $L(u, v, w)$ \eqref{L} in \eqref{Rearranged eqn}, we get
{\allowdisplaybreaks 
\begin{align*}
& \mathcal{F}(1;u,v,w) \nonumber\\
&= \log\left(1-u\right)\log\left(\frac{u(1-v)}{u-v}\right)\log\left(\frac{v}{v-w}\right) - \log\left(1-w\right) \operatorname{Li}_{2}\left(\frac{v(1-u)}{v-u}\right) \nonumber \\
&  - \log(1-u)\log\left(\frac{w(1-v)}{w-v}\right)\log\left(\frac{v(1-w)}{v-w}\right) - \log(1-u)\operatorname{Li}_{2}\left(\frac{v(1-w)}{v-w}\right) \nonumber \\
& - \log\left(\frac{v-u}{v-w}\right)\bigg[\operatorname{Li}_{2}\left(\frac{v(w-u)}{w(v-u)}\right) - \log \left(\frac{w}{w-u}\right)\log \left(\frac{u(w-v)}{w(u-v)}\right) \nonumber \\
& -\log\left(\frac{w(u-v)}{u(w-v)}\right) \log\left(\frac{v}{v-w}\right)\bigg] + \log\left(\frac{(v-u)(w-1)}{(v-w)(u-1)}\right)\bigg[\operatorname{Li}_{2}\bigg(\frac{(v-u)(w-1)}{(v-w)(u-1)}\bigg)  \nonumber \\
& +\operatorname{Li}_{2}\left(\frac{v(1-u)}{v-u}\right) -\operatorname{Li}_2\left(\frac{v(1-w)}{v-w}\right) -\operatorname{Li}_{2}\left(\frac{u(1-w)}{w(1-u)}\right) \nonumber \\
& -\log\left(\frac{w(u-v)}{u(w-v)}\right)\log\left(\frac{v(1-w)}{v-w}\right) \bigg] + \frac{1}{2}\log\left(\frac{w(u-v)}{u(w-v)}\right)\bigg[\log^{2} \left( \frac{u-v}{v(u-1)}\right) \nonumber \\
& - \log^2\left(\frac{v}{v-u}\right) 
  - 2\log\left(\frac{v}{v-w}\right)\log\left(\frac{v-u}{v-w}\right) + 2\log\left( \frac{v(1-w)}{v-w}\right)\log\left(\frac{(v-u)(w-1)}{(v-w)(u-1)}\right) \bigg] \nonumber\\
& + \operatorname{Li}_3\left(\frac{v(1-w)}{v-w}\right)-\operatorname{Li}_3\left(\frac{v}{v-w}\right) +\operatorname{Li}_3\left(\frac{v(1-u)}{v-u}\right)-\operatorname{Li}_3\left(\frac{v}{v-u}\right) \nonumber \\
& + \operatorname{Li}_{3}\left(\frac{u(1-w)}{w(1-u)}\right) - \operatorname{Li}_{3}\left(\frac{(v-u)(w-1)}{(v-w)(u-1)}\right) - \operatorname{Li}_{3}\left(\frac{u}{w}\right) + \operatorname{Li}_{3}\left(\frac{v-u}{v-w}\right).
\end{align*}
We simplify further to obtain
\begin{align*}
& \mathcal{F}(1;u,v,w) \nonumber\\
& = \log\left(1-u\right)\log\left(\frac{u(1-v)}{u-v}\right)\log\left(\frac{v}{v-w}\right) - \log\left(1-w\right) \operatorname{Li}_{2}\left(\frac{v(1-u)}{v-u}\right) \nonumber \\
&  - \log(1-u)\log\left(\frac{w(1-v)}{w-v}\right)\log\left(\frac{v(1-w)}{v-w}\right) - \log(1-u)\operatorname{Li}_{2}\left(\frac{v(1-w)}{v-w}\right) \nonumber \\
& - \log\left(\frac{v-u}{v-w}\right)\bigg[\operatorname{Li}_{2}\left(\frac{v(w-u)}{w(v-u)}\right) - \log \left(\frac{w}{w-u}\right)\log \left(\frac{u(w-v)}{w(u-v)}\right) \bigg] \nonumber \\
& + \log\left(\frac{(v-u)(w-1)}{(v-w)(u-1)}\right)\bigg[\operatorname{Li}_{2}\bigg(\frac{(v-u)(w-1)}{(v-w)(u-1)}\bigg) + \operatorname{Li}_{2}\left(\frac{v(1-u)}{v-u}\right) \nonumber \\
& -\operatorname{Li}_2\left(\frac{v(1-w)}{v-w}\right) -\operatorname{Li}_{2}\left(\frac{u(1-w)}{w(1-u)}\right)\bigg] \nonumber \\ 
& + \frac{1}{2}\log\left(\frac{w(u-v)}{u(w-v)}\right)\bigg[\log^{2} \left( \frac{u-v}{v(u-1)}\right) - \log^2\left(\frac{v}{v-u}\right)\bigg] \nonumber \\
& + \operatorname{Li}_3\left(\frac{v(1-w)}{v-w}\right)-\operatorname{Li}_3\left(\frac{v}{v-w}\right) +\operatorname{Li}_3\left(\frac{v(1-u)}{v-u}\right)-\operatorname{Li}_3\left(\frac{v}{v-u}\right) \nonumber \\
& + \operatorname{Li}_{3}\left(\frac{u(1-w)}{w(1-u)}\right) - \operatorname{Li}_{3}\left(\frac{(v-u)(w-1)}{(v-w)(u-1)}\right) - \operatorname{Li}_{3}\left(\frac{u}{w}\right) + \operatorname{Li}_{3}\left(\frac{v-u}{v-w}\right).
\end{align*}}
This finishes the proof of Theorem \ref{Final lemma}.
\end{proof}

\begin{proof}[Theorem \rm{\ref{val at rational for anlo}
}][]
Employing second part \eqref{special Value_5} of Theorem \ref{Dup: analogue}, we have
\begin{align*}
\mathcal{F}\left(\frac{p}{q};u,v,w\right)=\sum_{\alpha^q=1}\mathcal{F}(p;u,v^\frac{1}{q}\alpha,w).
\end{align*}
Again using \eqref{special Value_4} for $\mathcal{F}(p;u,v^\frac{1}{q}\alpha,w)$, we see that
\begin{align*}
\mathcal{F}\left(\frac{p}{q};u,v,w\right)&=\sum_{\alpha^q=1} \sum_{\beta^p=1} \sum_{\gamma^p=1} \mathcal{F}(1;u^\frac{1}{p}\beta,v^\frac{1}{q}\alpha,w^\frac{1}{p}\gamma),
\end{align*}
where $\mathcal{F}(1;u^\frac{1}{p}\beta,v^\frac{1}{q}\alpha,w^\frac{1}{p}\gamma)$ can be evaluated from Theorem \ref{Final lemma}.
\end{proof}

\begin{proof}[Corollary \rm{\ref{anal at natural}}][]
By putting $(p, q) = (n, 1)$ and $(p, q) = (1, n)$ in Theorem \ref{val at rational for anlo}, one can directly get the result. 
\end{proof}

\begin{proof}[Theorem \rm{\ref{theorem:analytic_continuation}}][]
Using $\frac{1}{1-vt}$ as the series expansion around $t=0$ and expanding the logarithm as series in \eqref{AHZN}, we get
\begin{align*}
\mathcal{F}_{k}(z;u,v) &= \int_{0}^{1}\left( \sum_{m=1}^{\infty}\frac{(ut^z)^m}{m}\right)^k \sum_{l=0}^{\infty} v^{l+1}t^l \, \text{d}t  \\
&=\int_{0}^{1} \sum_{m_1=1}^{\infty}\frac{(ut^z)^{m_1}}{m_1}\sum_{m_2=1}^{\infty}\frac{(ut^z)^{m_2}}{m_2} \cdots \sum_{m_k=1}^{\infty}\frac{(ut^z)^{m_k}}{m_k} \sum_{l=0}^{\infty} v^{l+1} t^l \text{d}t. 
\end{align*}
On switching the integration and summation order,  it gives
\begin{align*}
\mathcal{F}_{k}(z;u,v) &= \mathop{\prod_{i=1}^{k}\sum_{m_i=1}^{\infty}\sum_{l=0}^{\infty}}\frac{u^{m_i}}{m_i} v^{l+1}\int_{0}^{1} t^{zm_1+zm_2+\cdots +zm_k+l}dt,\\
&= \mathop{\prod_{i=1}^{k}\sum_{m_i=1}^{\infty}\sum_{l=0}^{\infty}}\frac{u^{m_i}}{m_i} \frac{ v^{l+1}}{(z(m_1+m_2+m_3+ \cdots +m_k)+l+1)},\\
&=  \mathop{\prod_{i=1}^{k}\sum_{m_i=1}^{\infty}\sum_{l=1}^{\infty}}\frac{u^{m_i}}{m_i}  \frac{v^{l}}{(z(m_1+m_2+m_3+ \cdots +m_k)+l)},
\end{align*} 
which is absolutely convergent for $|u|<1$ and $|v|<1$, and for any complex $z$ except for negative rational numbers.
\end{proof}

\begin{proof}[Theorem \rm{\ref{theorem:duplication_formula}}][]
From \eqref{AHZN}, we have
\begin{align}
\mathcal{F}_k\left(\frac{z}{n};u,v^n\right) = \int_{0}^{1}\frac{\log^k\left(1-ut^\frac{z}{n}\right)}{v^{-n}-t}dt.
\end{align}
Changing the variable $t$ by $e^{-t}$ in the above equation yields, 
\begin{align*}
\mathcal{F}_k\left(\frac{z}{n};u,v^n\right) &= \int_{0}^{\infty}\frac{\log^k\left(1-ue^{\frac{-tz}{n}}\right)}{e^{t}v^{-n}-1} \,dt.
 \end{align*}
 Making the change of variable $t=ny$ and using the relation \eqref{ABCD}
in the above equation, one obtains
 \begin{align*}
 \mathcal{F}_k\left(\frac{z}{n};u,v^n\right) =\sum_{\beta^n=1}\mathcal{F}_k(z;u,\beta v).
 \end{align*}
 This completes the proof of Theorem \ref{theorem:duplication_formula}.
\end{proof}

\begin{proof}[Theorem \rm{\ref{GEN at 1/n}}][]
From multiplication formula \eqref{special Value_3},  we have
\begin{align}
\mathcal{F}_k\left(\frac{1}{n};u,v\right)&=\sum_{\beta^n=1}
\mathcal{F}_k(1;u, \beta v^\frac{1}{n}).\nonumber
\end{align}
Substituting value of $\mathcal{F}_k(1;u,\beta v^\frac{1}{n})$ from Lemma \ref{f11}, one can directly obtain the result.
\end{proof}

\begin{proof}[Corollary \rm{\ref{Particular for GEN at 1/n}}][]
Substituting $k=1$ in Theorem \ref{GEN at 1/n}, we have
\begin{align}
\mathcal{F}_1\left(\frac{1}{n};u,v\right) &= \sum_{\beta^n = 1} \bigg\{\log(1-u) \textup{Li}_1 \left(\frac{\beta v^{\frac{1}{n}}(u-1)}{u - \beta v^{\frac{1}{n}}} \right) - \textup{Li}_2 \left(\frac{\beta v^{\frac{1}{n}}(u-1)}{u - \beta v^{\frac{1}{n}}} \right) \nonumber \\
& + \textup{Li}_2 \left(\frac{\beta v^{\frac{1}{n}}}{\beta v^{\frac{1}{n}} - u}\right)\bigg\}.
\end{align}
One can see that $\mathcal{F}_1\left(z;u,v\right) = \mathcal{F}\left(z;u,v\right)$. Now we state the following Abel's identity \cite[Eq.~(1.16)]{Lewin}, which is also equivalent to Roger's identity \eqref{dilog fn eqn}:
\begin{align*}
\operatorname{Li}_{2}\left(\frac{x}{1-y}\right) + \operatorname{Li}_{2}\left(\frac{y}{1-x}\right) &- \operatorname{Li}_{2}\left(\frac{xy}{(1-x)(1-y)}\right) \\
& = \operatorname{Li}_{2}\left(x\right) + \operatorname{Li}_{2}\left(y\right) + \log(1-x)\log(1-y).
\end{align*}
Use the above identity with $x=u$, $y = \frac{\beta v^{\frac{1}{n}}(u-1)}{u - \beta v^{\frac{1}{n}}}$ and simplify to get \eqref{RC at 1/n}.
\end{proof}

\begin{proof}[Theorem \ref{u=v and x=1}][]
From the definition \eqref{AHZN} of $\mathcal{F}_k\left(z;u,v\right)$, we have
\begin{align}\label{dk}
\mathcal{F}_k\left(1;u,u\right)= \int_{0}^1 \frac{\log^k(1-ut)}{u^{-1}-t} \mathrm{d}t. 
\end{align}
Note that $$
\frac{1}{k+1}\frac{\textrm{d}}{\textrm{d}t} \log^{k+1}(1-ut)= - \frac{\log^k(1-ut)}{u^{-1}-t}.
$$
Employing the above identity in \eqref{dk}, one can easily get the desired result. 
\end{proof}

\section{Concluding Remarks}
Throughout this paper, our primary focus has been to study the properties of the functions $\mathcal{F}(z;u,v,w)$ and $\mathcal{F}_k(z;u,v)$, defined in \eqref{AHZN-1} and \eqref{AHZN}.
Mainly, we have established multiplication formula, analytic continuation and special values at rational arguments of these functions. It is quite challenging to obtain all the similar properties for our functions that $\mathcal{F}(z;u,v)$ satisfies.  Investigating other properties of these functions would be a compelling direction for future research.

The Table \ref{Table of main theorem} below illustrates some special values of the function $\mathcal{F}_k(z;u,v)$ evaluated from Theorem \ref{GEN at 1/n}.
\newpage
\begin{table}[h]
\caption{Special values of $\mathcal{F}_{k}\left(\frac{1}{n};u,v\right)$}
		\label{Table of main theorem}
		\renewcommand{\arraystretch}{1}
		{
\begin{tabular}{|l|l|l|l|}	
\hline
$\left(k, n\right)$ & $(u,v)$  & Value of $\mathcal{F}_{k}\left(\frac{1}{n};u,v\right)$  \\
\hline  
$(1, 1)$ & $\left(\frac{1}{2}, -1\right)$  & $\log{(2)} \log{\left(\frac{2}{3}\right)}-\text{Li}_{2}\left(\frac{1}{3}\right)+\text{Li}_{2}\left(\frac{2}{3}\right)$\\
\hline
$(1, 1)$ & $\left(\frac{1}{3}, -1\right)$  & $\log{(2)} \log{\left(\frac{2}{3}\right)} -\text{Li}_{2}\left(\frac{1}{2}\right)+\text{Li}_{2}\left(\frac{3}{4}\right)$\\
\hline
$(1,1)$ & $\left(1, \frac{1}{4}\right)$  & $\text{Li}_{2}\left(-\frac{1}{3}\right)$\\      
\hline
$(1, 1)$  &$(-4, -2)$  & $ -\log(5) \log(6) -\text{Li}_{2}(-5)  - \frac{\pi^2}{12}  $\\
\hline
$(2, 2)$ 	& $ \left(1, \frac{1}{9} \right)$  & $-2 \left[ \text{Li}_{3}\left(-\frac{1}{2}\right) + \text{Li}_{3}\left(\frac{1}{4}\right) \right]$\\
\hline
$(2, 1)$ 	& $\left(-2, -3\right)$  & $-\log^2(3) \log (8) - i \pi\log^2(3) - 2\log(3) \text{Li}_{2}(9) + 2 \text{Li}_{3}(9)-2 \text{Li}_{3}(3)  $\\
\hline
$(1, 2)$ & $\left(1,-1\right)$ & $\text{Li}_{2}\left(\frac{i}{i-1}\right) + \text{Li}_{2}\left(\frac{i}{i+1}\right)$\\
\hline
$(2, 1)$	& $(1, -1)$  & $-2\, \text{Li}_{3}\left(\frac{1}{2}\right)$\\
\hline
\end{tabular}}

\end{table}	

\section{Acknowledgements}
The authors want to thank the anonymous referee for carefully reading the manuscript.
The first author’s research is supported by the Prime Minister Research Fellowship (PMRF), Govt. of India, Grant
No. 2102227. The second author wants to thank Science and Engineering Research Board (SERB), India, for giving MATRICS grant (File No. MTR/2022/000545) and SERB CRG grant (File No. CRG/CRG/2023/002122).

\end{document}